\documentclass[12pt,oneside,reqno]{amsart}

\usepackage{amsmath, amsthm, amsfonts, amsfonts,amssymb,amscd,amsmath,latexsym,enumerate,verbatim, calc, quotmark, mathtools, txfonts, calrsfs, imakeidx, fancyhdr, listings, lipsum, tkz-euclide, amsthm}%, setspace}authblk

\usepackage{mathtools}
\DeclarePairedDelimiter{\ceil}{\lceil}{\rceil}

\usepackage[T1]{fontenc}
\usepackage[utf8]{inputenc}

\usepackage[all, cmtip]{xy}

\usepackage[pagebackref]{hyperref}
\hypersetup{
     colorlinks = true, linkcolor = red,
     citecolor   = blue,
     urlcolor    = blue,
}

\newtheorem{innercustomgeneric}{\customgenericname}
\providecommand{\customgenericname}{}
\newcommand{\newcustomtheorem}[2]{%
  \newenvironment{#1}[1]
  {%
   \renewcommand\customgenericname{#2}%
   \renewcommand\theinnercustomgeneric{##1}%
   \innercustomgeneric
  }
  {\endinnercustomgeneric}
}

\newcustomtheorem{customthm}{Theorem}
\newcustomtheorem{customqn}{Question}
\def\NZQ{\mathbb}               % the font for N,Z,Q,R,C

\def\Z{{\NZQ Z}}

\def\m{\mathfrak{m}}
\def\n{\mathfrak{n}}
\def\p{\mathfrak{p}}
\def\q{\mathfrak{q}}

\def\i{\mathfrak{i}}
\def\j{\mathfrak{j}}

\DeclareMathOperator*{\Hom}{Hom}
\DeclareMathOperator*{\Ho}{H}

\DeclareMathOperator*{\depth}{depth}
\DeclareMathOperator*{\Tor}{Tor}
\DeclareMathOperator*{\Ext}{Ext}
\DeclareMathOperator*{\Zi}{Z}
\DeclareMathOperator*{\Bo}{B}

\DeclareMathOperator*{\Ch}{char}

\DeclareMathOperator*{\ann}{ann}
\DeclareMathOperator*{\soc}{socle}
\DeclareMathOperator*{\gr}{gr}
\DeclareMathOperator*{\lo}{ll}
\DeclareMathOperator*{\edim}{edim}
\DeclareMathOperator*{\mx}{max}
\DeclareMathOperator*{\I}{I}

\DeclareMathOperator*{\syz}{Syz}
\DeclareMathOperator*{\cn}{\mathfrak{c}}

\DeclareMathOperator*{\pd}{pd}
\DeclareMathOperator*{\cdepth}{codepth}

\newtheorem{theorem}{Theorem}[section]
\newtheorem{lemma}[theorem]{Lemma}
\newtheorem{corollary}[theorem]{Corollary}
\newtheorem{proposition}[theorem]{Proposition}
\newtheorem{remark}[theorem]{Remark}

\newtheorem{definition}[theorem]{Definition}
\newtheorem*{definition*}{Definition}
\newtheorem*{examples*}{Examples}

\newtheorem*{acknowledgement}{Acknowledgement}

\title{A criterion for modules over Gorenstein local rings to have rational Poincar\'e series
}
\author{Anjan Gupta}

\address{Department of Mathematics\\
Indian Institute of Science Education and Research Bhopal\\
Bhopal Bypass Road, Bhopal, Madhya Pradesh, India. Pin - 462066.}
\email{agmath@gmail.com}

\thanks{Corresponding author: Anjan Gupta; 
{\it email: agmath@gmail.com}}

\begin{document}
\begin{abstract} 
We prove that modules over an Artinian Gorenstein local ring $R$ have rational Poincar\'e series sharing a common denominator if  $R/\soc(R)$ is a Golod ring. If $R$ is a Gorenstein local ring with square of the maximal ideal being generated by at most two elements, we show that modules over $R$ have rational Poincar\'e series sharing a common denominator. By a result of \c Sega, it follows that $R$ satisfies the Auslander-Reiten conjecture. 
We provide a different proof of a result of Rossi and \c Sega \cite{rossi} concerning rationality of Poincar\'e series of modules over compressed Gorenstein local rings.  
We also give a new proof of the fact that modules over Gorenstein local rings of codepth at most three have rational Poincar\'e series sharing a common denominator, which is originally due to Avramov, Kustin and Miller \cite{av4}.

\end{abstract}
\maketitle 



\section{introduction}
Let $R$ be a commutative Noetherian local ring with maximal ideal $\m$ and residue field $k = R / \m$. Let $M$ be a finitely generated module over $R$.  The Poincar\'e series of $M$ over $R$ is a formal power series in $\Z[|t|]$ defined as follows: 
\[P^R_M(t) = \sum_{i \geq 0} \beta^R_i(M) t^i \in \Z[|t|],\]
where $\beta^R_i(M) =  \dim_k \Tor^R_i(M, k)$ denotes the $i$-th Betti number of $M$. 
We say that a formal power series $P(t) \in \Z[|t|]$ is a rational function if there exists a polynomial $g(t) \in \Z[t]$ such that $g(t) P(t)$ is a polynomial in $\Z[t]$.
An example due to Anick shows that the Poincar\'e series $P^R_k(t)$ is not a rational function in general, cf. \cite{anick}.
B\o gvad observed that $P^R_k(t)$ may not be a rational function even if $R$ is a Gorenstein ring, cf. \cite{bogvad}.

Following Roos we say that a ring R is good if there exists a polynomial $d_R(t) \in \Z[t]$ such that $d_R(t)P_M^R (t) \in \Z[t]$ for every finitely generated $R$-module $M$ and bad otherwise, cf. \cite[Definition 2.1]{roos2}. 
Roos proved that bad rings exist, cf. \cite[Theorem 2.4]{roos2}. Nevertheless there are an abundance of good rings, e.g. regular local rings, local complete intersections, cf. \cite[Corollary 4.2]{gul1}. We refer to \cite{av4}, \cite{av2} for more examples of good rings and a detailed account  of applications of rationality of Poincar\'e series. 

In this paper, we use $\mu(-)$ to denote the minimal number of generators. The embedding dimension of $R$ ($= \dim_k {\m/ \m^2}$) is denoted by $\edim(R)$. Let $\hat{R}$ denote the $\m$-adic completion of $R$. 
%Note that $\hat{R} = R$ if $R$ is an Artinian ring. 
By Cohen's structure theorem, there is a regular local ring $Q$ with maximal ideal $\n$ and a surjective ring homomorphism $\eta : Q \rightarrow \hat{R}$ such that $\ker \eta = I \subset \n^2$. The map $\eta$ is called a minimal Cohen presentation of $R$. The Loewy length of $R$ is defined as $\lo(R) = \max\{i : \m^i \neq 0\}$ if $R$ is Artinian and infinity otherwise. 

We recall a few more examples of good rings collected from existing literature. Precise reference is given with each of the examples.

\begin{examples*} 
 Let $R$ be a Gorenstein local ring such that $\edim(R) = n \geq 2$, $\lo(R) = s$ and $\mu(I) = r$. If $R$ satisfies one of the conditions (1)-(3) below, then $P^R_k(t) = \frac{(1 + t)^n}{d_R(t)}$ for some polynomial $d_R(t) \in \Z[t]$ and $d_R(t)P_M^R (t) \in \Z[t]$ for every finitely generated $R$-module $M$.

\begin{enumerate}
\item
$R$ is a compressed Artinian Gorenstein ring 
%that is R has the maximal length among all local Artinian Gorenstein rings with the same embedding dimension and Loewy length, 
(see Definition \ref{comp}) and $s \geq 2, s \neq 3$. If $\eta : Q \rightarrow R$ is a minimal Cohen presentation of $R$, then the polynomial $d_R(t)$ is given by $1 - t(P^Q_R(t) - 1) + t^{n+1}(1+t)$; cf. Rossi and \c Sega \cite[Theorem 5.1]{rossi},

\item
$R$ is an Artinian Gorenstein ring and $\mu(\m^2) = 1$. The polynomial $d_R(t)$ is given by $1 - nt + t^2$; 
cf. Sally \cite[Theorem 2]{sally}, Croll et al. \cite[Theorem 5.4]{amenda},

\item
$R$ is not a complete intersection and $\cdepth(R) = \edim(R) - \depth(R) \leq 3$. The polynomial $d_R(t)$ is equal to $1 - rt^2 - rt^3 + t^5$; cf. Wiebe \cite[Satz 9]{wiebe}, Avramov et al. \cite[Theorem 6.4]{av4}.
\end{enumerate}
%Then, the Poincar\'e series is given by \(P^R_k(t) = \displaystyle \frac{(1 + t)^n}{d_R(t)}\) for some polynomial $d_R(t) \in \Z[t]$ and $d_R(t)P_M^R (t) \in \Z[t]$ for every finitely generated $R$-module $M$.
\end{examples*}
The main objective of the present article is to give a criterion for Gorenstein local rings to be good, which provides a common method to prove the good property in each of the above examples. As a new application we show that if $R$ is an Artinian Gorenstein ring and $\mu(\m^2) = 2$, then $R$ is a good ring.

We recall a few definitions. 
Let $\phi: R \rightarrow S$ be a surjective homomorphism of local rings and $k$ be the common residue field of $R$ and $S$.
From the standard change of rings spectral sequence of $\Tor$%$\Tor^S_p(k, \Tor^R_q(S, k)) \underset{p}{\Rightarrow} \Tor^R_{p + q}(k, k)$
, Serre proved the following term-wise inequality of power series:
\[
P^S_k(t) \prec \frac{P^R_k(t)}{1 - t(P^R_S(t) - 1)}.
\]
The homomorphism $\phi$ is called a Golod homomorphism if the above inequality is an equality.
% We say a local ring $R$ has the Backelin-Roos property if there is a surjective Golod homomorphism from a complete intersection onto $R$. 
 The most widespread method to show that a ring $R$ is good is to use 
 a result of Levin (Theorem \ref{prl9.5}) which states that a ring $R$ is good if there is a surjective Golod homomorphism from a complete intersection onto $R$.

Let $\eta : Q \twoheadrightarrow \hat{R}$ be a minimal Cohen presentation of $R$. We say that $R$ is a Golod ring if $\eta$ is a Golod homomorphism. Let $\edim(R) = n$ and $K^R$ denote the Koszul complex of $R$ on a minimal set of generators of maximal ideal $\m$. It follows that $R$ is a Golod ring whenever one has \(P^R_k(t) = \displaystyle \frac{(1 + t)^n}{1 - \sum_{i = 1}^n \dim_k \Ho_i(K^R) t^{i+1}}\). We refer to \cite[\S 3]{av} for more details on Golod rings and Golod homomorphisms. 
The main result of the present article is the following:

\begin{customthm}{I} \label{intro2}
Let $R$ be an Artinian Gorenstein local ring of embedding dimension $n \geq 2$ such that $R/ \soc(R)$ is a Golod ring. Let $\eta : Q \rightarrow R$ be a minimal Cohen presentation, $\n$ denote the maximal ideal of $Q$ and $I = \ker(\eta) \subset \n^2$. Then the following hold.

\begin{enumerate}
\item
For any $f \in I \setminus \n I$, the induced map $Q/ (f) \twoheadrightarrow R$ is a Golod homomorphism.
% i.e. $R$ has the Backelin-Roos property.

\item
Let $d_R(t) =1 - t(P^Q_R(t) - 1) + t^{n+1}(1+t)$. Then for any $R$-module $M$ we have $d_R(t) P^R_M(t) \in \Z[t]$.
\end{enumerate}
\end{customthm}

It is worth noting that if $R$ is an Artinian Gorenstein ring, $\edim(R) \geq 2$ and $R/ \soc(R)$ is a Golod ring, then with the notations used in the above theorem, one has  $P^R_k(t) = \displaystyle \frac{(1 + t)^n}{d_R(t)}$ by a result of Rossi and \c Sega, cf. \cite[Proposition 6.2]{rossi}. Therefore, the statement (2) is an immediate consequence of statement (1) and the result of Levin.

The following is proved as an application of Theorem \ref{intro2}.

\begin{customthm}{II}\label{intro3.5}
Let $R$ be an Artinian Gorenstein local ring with maximal ideal $\m$ and residue field $k$. Let $M$ be a finitely generated $R$-module. Assume that $\edim(R) = n$ and $\mu(\m^2) \leq 2$. 
%Let $\eta : Q \twoheadrightarrow R$ be a minimal Cohen presentation, $\n$ be the maximal ideal of $Q$ and $I = \ker \eta \subset \n^2$. 
Then the following hold.

\begin{enumerate}
\item
If $n= 1$, then $\displaystyle P^R_k(t) = \frac{1}{1 - t}$ and $(1- t)P^R_M(t) \in \Z[t]$,

\item
If $n \geq 2$, then $\displaystyle P^R_k(t) = \frac{1}{1 - nt + t^2}$ and $(1+t)^{n}(1 - nt + t^2)P^R_M(t) \in \Z[t]$,

\item
If $\Ext^i(M, M) = 0$ for all $i \geq 1$, then $M$ is a free $R$-module.
\end{enumerate}
\end{customthm}
The statement (3) follows from  statements (1) and (2) by an argument of \c Sega, cf. \cite{sega1}. It implies that $R$ satisfies the Auslander-Reiten conjecture, cf. \cite{auslan}.

The rings considered in Theorem \ref{intro3.5} are called stretched when $\mu(\m^2) = 1$, respectively almost stretched when $\mu(\m^2) = 2$ (see Definition \ref{intro0}). Stretched Cohen-Macaulay local rings were introduced by Sally in \cite{sally}.  She proved that $P^R_k(t)$ is rational for such a ring $R$, cf. \cite[Theorem 2]{sally}. Later  Elias and Valla introduced almost stretched Cohen-Macaulay local rings. 
%introduced almost stretched Cohen-Macaulay local rings in \cite{elias} and 
They proved  that if  $R$ is an almost stretched Gorenstein local ring and the residue field $k$ of $R$ has characteristic zero, then $P^R_k(t)$ is rational, cf. \cite[Theorem 1.1]{elias1}. 
In a recent article \cite[Corollary 5.6]{amenda}, stretched Cohen-Macaulay local rings are shown to be good. Using Theorem \ref{intro3.5}, we prove that stretched Cohen Macaulay and almost stretched Gorenstein rings are good without any assumption on residue fields. We also prove that such rings satisfy the Auslander-Reiten conjecture.

Now we briefly describe the organisation of the article. In \S 2, we prove Theorem \ref{intro2}. The proof extensively uses a characterisation theorem for Golod algebras (see Theorem \ref{prl7}) and chain derivations on acyclic closures whose construction dates back to the work of Gulliksen.
The connected sum of Gorenstein local rings 
%with a common residue field $k$ 
is introduced in \cite{an3} (see Definition \ref{prl10}). In \S 3, we provide a criterion for connected sum decompositions of Gorenstein local rings. We show that if $R$ is an Artinian Gorenstein local ring with maximal ideal $\m$ and $\mu(\m^2) \leq 2$, then $R$ decomposes as a connected sum unless $\mu(\m) \leq 2$ (Corollary \ref{ascs}). We use this decomposition to show that quotients of such a ring $R$ by non zero powers of maximal ideal $\m$ are Golod rings (Lemma \ref{ds4}). This fact is crucially used in the proof of Theorem \ref{intro3.5}.
The final section \S 4 contains new proofs of Examples (1) and (3) using Theorem \ref{intro2}. We identify a certain quotient $C$ of the Koszul algebra $K^R$ of an Artinian Gorenstein ring $R$ such that $R$ is a surjective image of a complete intersection under a Golod homomorphism whenever $C$ is a Golod DG algebra (\S \ref{prl4}). We show that for the ring considered in examples (1) and (3), this quotient is a Golod algebra. We make it a point to advertise here that our versions are slightly stronger than the earlier ones in both examples since we constructed Golod homomorphisms from hypersurfaces given by any choice of generator belonging to a minimal generating set of the defining ideal. 
%By a similar approach, Theorem \ref{sbr6} follows. 

We conclude with a remark that our approach only constructs Golod homomorphisms from hypersurface rings. We hope that the present approach can be generalised further to find criteria for existence of Golod homomorphisms from complete intersections of higher codimension.

All rings in this article are Noetherian local rings with $1\not=0$. All modules are nonzero and finitely generated.
Throughout this article, the expression ``local ring $(R,\m, k)$" refers to a commutative Noetherian local ring $R$ with maximal ideal $\m$ and residue field $k = R/\m$. When information on the residue field is not necessary, we denote a local ring $R$ with maximal ideal $\m$ simply by $(R, \m)$.

%%%%%%%%%%%%%%%%%%%%%%%%%%%%%%%%%%%%%%%%%%%%%%%%%%%%%%%%%%%%%%%%%

\section{The main result}
Let $(R, \m, k)$ be a local ring. A DG algebra $(A, \partial)$ over the ring $R$ consists of a non-negatively graded strictly skew-commutative $R$-algebra $A = \oplus_{i \geq 0} A_i$ such that $ A_0 = R/ I$ for some ideal $I$ of $R$ and an $R$-linear differential map $\partial$ of degree $-1$ satisfying Leibniz rule. A DG-algebra homomorphism $\phi : A \rightarrow B$ is a chain map of complexes which induces a ring homomorphism $\phi^{\#} : A^{\#} \rightarrow B^{\#}$ between underlying skew-commutative rings $A^{\#}$, $B^{\#}$ after forgetting the differential maps on $A$ and $B$. The DG-algebra $B$ is called a semi-free extension of $A$ if $B^{\#}$ is a free module over $A^{\#}$. 

The DG algebra $(A, \partial)$ is augmented if it is equipped with a surjective DG algebra homomorphism $\epsilon : A \twoheadrightarrow k$. If $\tilde{\epsilon} : \Ho(A) \rightarrow k$ is the induced map on homology, we set $\I A = \ker {\epsilon}$, $\I\Ho(A) = \ker \tilde{\epsilon}$ and $\I \Zi(A) = \I A \cap \Zi(A)$. 
We say that the DG algebra $(A, \partial)$ is minimal if $\partial(A) \subset \m A$. A minimal DG algebra $A$ is augmented naturally with the surjective map $\epsilon = q \circ pr$ where $pr : A \twoheadrightarrow A_0$ is the projection and $q: A_0 \twoheadrightarrow k$ is the natural quotient map. We refer to \cite{av} for more information on DG algebras and related terminologies.

\subsection{Tate resolutions}
Tate described a method to construct a DG algebra resolution of the residue field $k$ over $R$. The method involves an iterated process of adjoining exterior variables to kill cycles of even degrees and divided powers variables to kill cycles of odd degrees  starting from $R$. In literature, this construction is known as Tate resolution.
Later Gulliksen proved that if the number of variables added at each step of killing cycles of a certain degree is the minimum possible, the resulting Tate resolution  becomes a minimal free resolution of the residue field $k$.  
In this case, the DG algebra is called the acyclic closure of $k$ over $R$ which is unique up to isomorphism of DG$\Gamma$ algebras. We refer the reader to  \cite[\S 6]{av}, \cite[Chapter 1]{gul} for more details.

In this article, by Tate resolution we mean a surjective $R$-linear quasi-isomorphism  $ \epsilon : R\langle X \rangle \twoheadrightarrow k$ where $R\langle X \rangle$ is the acyclic closure of $k$. The adjoined set of variables $X = \{X_i : i \geq 1\}$ is ordered such that $1 \leq \deg(X_i) \leq \deg(X_j)$ for $i< j$. Note that by construction the acyclic closure $R\langle X \rangle$ is a semi-free extension of $R$.

%satisfies the following: 
%\begin{enumerate}
%\item $X_i$ is a divided powers variable when $\deg(X_i)$ is even and an exterior variable when $\deg(X_i)$ is odd. Moreover $1 \leq \deg(X_i) \leq \deg(X_j)$ for $i< j$.
%
%\item 
%If $\partial$ denotes the differential on $R\langle X\rangle$, then $\{\partial(X_i)$ : $\deg(X_i) = 1\}$ is a minimal generating set of the maximal ideal $\m$ and homology classes of cycles in  $\{\partial(X_i) :  \deg(X_i) = n\}$ minimally generate $\Ho_{n-1}(R\langle X\rangle)$ for all $n \geq 2$.
%\end{enumerate}

An $R$-linear derivation of degree $n$ on the acyclic closure $R\langle X\rangle$ is an $R$-linear map $\eta : R\langle X\rangle \rightarrow R\langle X\rangle$ of degree $n$ satisfying the following properties :  

\begin{enumerate}
\item
$\eta(R) = 0$ ($R$-linearity);

\item
$\eta$ satisfies the Leibniz rule, i.e.  $\eta(uv) = \eta(u)v + (-1)^{n\deg(u)}x\eta(v)$ for $u, v \in R\langle X\rangle$;

\item 
$\eta(X_i^{(i)}) = \eta(X_i) X_i^{(i-1)}$, $X_i^{(i)}$ being the $i$-th divided power of a variable $X_i$ of  even positive degree.
\end{enumerate}
The derivation $\eta$ is called a chain derivation if it commutes with the differential $\partial$ of $R\langle X\rangle$ in the graded sense, i.e. $\eta \circ \partial = (-1)^n \partial \circ \eta$. 
 
Gulliksen in \cite[Theorem 1.6.2]{gul} constructed a sequence of $R$-linear chain derivations $\eta_j$ on the acyclic closure $R\langle X\rangle$ such that $\eta_j(X_j) = 1$ and $\eta_j(X_i) = 0$ for $i < j$. 

%Using these chain derivations, he proved that the Tate resolution $ \epsilon : R\langle X \rangle \twoheadrightarrow k$ is minimal, i.e. $\partial(R\langle X \rangle) \subset \m R\langle X \rangle$. The proof is also available in \cite[Lemma 6.3.3]{av}.

\subsection{Golod algebras}\label{prl4}
% If $W = \oplus_{i \geq 0}W_i$ is a graded $k$-vector space, we define the Hilbert series $\Ho_W(t) = \sum_{i \geq 0} \dim_k W_it^i \in \ZZ[|t|]$. Let $(R, \m, k)$ be a local ring. Let $A$ be a minimal DG algebra over $R$ such that $\Ho_0(A) = k$ and each $A_i$ is a free $R$-module of finite rank. Then there is a term-wise inequality of power series $P^R_k(t) \prec \frac{\Ho_{A \otimes k}(t)}{1 - t (\Ho_{\Ho(A)}(t) - 1)}$. Following Levin \cite[Chapter 1, \S 4]{lev3}, we define $A$ to be a Golod algebra when equality holds. 
 
An augmented DG algebra $A$ over $R$ with augmentation map $\epsilon : A \twoheadrightarrow k$ is called a Golod algebra if $A$ admits a trivial Massey operation, i.e. 
there is a graded $k$-basis $\mathfrak{b}_R = \{h_{\lambda}\}_{\lambda \in \Lambda}$ of $\I \Ho(A)$, a function $\mu : \sqcup_{i = 1}^{\infty} \mathfrak{b}_R^i \rightarrow  A$ such that $\mu(h_\lambda) \in \I \Zi(A)$ with  $cls(\mu(h_\lambda))= h_\lambda$ and setting $\bar{a} = (-1)^{i+1}a$ for $a \in A_i$ one has 
\[ \partial \mu(h_{\lambda_1}, \ldots ,h_{\lambda_p}) = \sum_{j = 1}^{p-1}\overline{\mu(h_{\lambda_1}, \ldots ,h_{\lambda_j})}\mu(h_{\lambda_{j + 1}}, \ldots ,h_{\lambda_p}).\]

The following is proved in \cite[Theorem 1.5]{lev3} and also follows from \cite[Theorem 1.1]{lev5}. 

\begin{theorem}\label{prl7}
Let $f : (R, \m) \twoheadrightarrow (S, \n)$ be a surjective homomorphism of local rings with common residue field $k$ and $\epsilon : R\langle X \rangle \twoheadrightarrow k$ be a DG algebra resolution of $k$ over $R$. Set 
$A = R\langle X \rangle \otimes_R S$. Consider $A$ augmented with the augmentation $\epsilon \otimes_R S$. Then the following are equivalent:

\begin{enumerate}
\item
The DG algebra $A$ is a Golod algebra.

\item
The map $f$ is a Golod homomorphism.

\item
The induced maps $\Tor^R(k, k) \rightarrow \Tor^S(k, k)$ and $\Tor^R(\n, k) \rightarrow \Tor^S(\n, k)$ are injective. 
\end{enumerate}
\end{theorem}

We recall the following result of Levin recorded in \cite[Proposition 5.18]{av4}.
\begin{theorem}\label{prl9.5}
Let $(R, \m, k)$ be a local ring and $\phi : P \twoheadrightarrow R$ be a surjective Golod homomorphism from a local complete intersection $P$ of embedding dimension $n$ onto $R$.  Then there exists a polynomial $d_R(t) \in \Z[t]$ such that for any finitely generated $R$-module $M$, we have $d_R(t) P^R_M(t) \in \Z[t]$. Moreover, $d_R(t) P^R_k(t) = (1 + t)^n$.
\end{theorem}

%\begin{lemma}\label{prl9.5}
%Let $(R, \m)$ be an Artinian Gorenstein local ring of embedding dimension $n$. Let $\soc(R) = (s)$ and $K^R$ denote the Koszul complex of $R$. Then $s K^R_i \subset (0 : \m^2) K^R_{i + 1}$ for $0 \leq i \leq n - 1$.
%\end{lemma}
%
%
We are now equipped to prove the main result.

\subsection{Proof of Theorem \ref{intro2}}
\begin{proof}
Let the maximal ideal of $R$ be $\m$ and  $k = R/ \m$ denote the residue field of $R$. 
We know that $\Ho_1(K^{R}) \cong I/ \n I$. Therefore, the minimal generators of $I$  are in one-one correspondence with the generators of $\Ho_1(K^{R})$. Let the maximal ideal $\n$ of $Q$ be minimally generated by $y_1, \ldots, y_n$. The Koszul complex of $Q$ is $K^Q = Q \langle X_i : \partial(X_i) = y_i, 1 \leq i \leq n \rangle$. Let $f = \sum_{i=1}^n a_iy_i$ and $P = Q/ (f)$.
Note that $K^P = K^Q \otimes_Q P$, $K^{R} = K^Q \otimes_Q {R}$ are Koszul complexes of $P$, ${R}$ respectively. Set $z = \sum_{i=1}^n a_iX_i \in K^Q_1$. Then its residue class $\bar{z}$ is a cycle in $\Zi_1(K^P)$. Let  $V = K^P\langle T : \partial(T) = \bar{z} \rangle$ be the extension of $K^P$ by adjoining a divided powers variable $T$ of degree two to kill the cycle $\bar{z}$.
By \cite[Theorem 7.3.3]{av}, the natural augmentation $V \twoheadrightarrow k$ is the Tate resolution of $k$ over $P$. 

Set $U = V \otimes_R R = K^R\langle T : \partial(T) = \bar{z} \rangle$. Since $f \in I \setminus \n I$, we have $\bar{z} \in \Zi_1(K^{R}) \setminus B_1(K^{R})$. Therefore, we can adjoin variables to $U$ to obtain the acyclic closure $\mathfrak{X}$ of the residue field $k$ over $R$. The augmentation map $\epsilon : \mathfrak{X} \twoheadrightarrow k$ is the Tate resolution of $k$ over $R$.

 By Theorem \ref{prl7}, to show that $P \twoheadrightarrow R$ is a Golod homomorphism, we need to prove that the induced maps $\Tor^P(k, k) \rightarrow \Tor^R(k, k)$ and $\Tor^P(\m, k) \rightarrow \Tor^R(\m, k)$ are injective. Both $V, \mathfrak{X}$ are minimal algebras. Therefore, the first map is $U\otimes_R k \rightarrow \mathfrak{X} \otimes_R k$ which is obviously injective since $\mathfrak{X}$ is a semi-free extension of $U$.  
 The second map is $i_* : \Ho(\m U) \rightarrow \Ho(\m \mathfrak{X})$ which is induced by the 
inclusion $\i : \m U \rightarrow \m \mathfrak{X}$. We prove that $\i_*$ is an injective map.

We have an $R$-linear chain derivation $v : \mathfrak{X} \rightarrow \mathfrak{X}$ of degree $-2$ such that $v(T) = 1$.
Set $\bar{R} = R / \soc(R)$. Note that $\soc(R) \subset \m^2$, so $K^{\bar{R}} = \bar{R} \otimes_R K^R$ is the Koszul complex of $\bar{R}$. Now $K^{\bar{R}}$ can be extended to the acyclic closure $\mathfrak{Y}$ over $\bar{R}$. Let $\j : K^{\bar{R}} \rightarrow \mathfrak{Y}$ denote the inclusion. 
The augmentation $\epsilon_{\mathfrak{Y}} : \mathfrak{Y} \twoheadrightarrow k$ is an algebra homomorphism over $K^R$. The acyclic closure $\mathfrak{X}$ is semi-free over $K^R$. Therefore, the augmentation $\epsilon_{\mathfrak{X}} : \mathfrak{X} \twoheadrightarrow k$ lifts to a DG algebra homomorphism $\beta : \mathfrak{X} \rightarrow \mathfrak{Y}$ over $K^R$ , cf.\cite[Proposition 2.1.9]{av}. Let  $\alpha : K^R \rightarrow K^{\bar{R}}$ denote the quotient map. By abuse of notation we denote restriction of a map by the same symbol.
We have the following commutative diagram.

\[\xymatrix{
\m K^R \ar[d]^{\alpha} \ar[r]^\i &\m \mathfrak {X}\ar[d]^{\beta}\\
\m K^{\bar{R}} \ar[r]^\j          &\m \mathfrak{Y}}\]

A cycle $y$ in $\m U$ can be written as $y = \sum_{k=0}^m a_k T^{(m-k)}$, $a_i \in \m K^R$. Suppose $\i(y)$ is in the boundary of $\m \mathfrak{X}$. We prove by induction on $m$ that $y$ is in the boundary of $\m U$.

First assume that $m = 0$. Then $y \in \m K^R$. Since $\i(y)$ is in the boundary of $\m \mathfrak{X}$, $\j \circ \alpha(y)$ is in the boundary of $\m \mathfrak{Y}$ by the commutative diagram. Now $\bar{R}$ is a Golod ring, so $\j$ induces an injective map $\j_* : \Ho(\m K^{\bar{R}}) \rightarrow \Ho(\m \mathfrak{Y})$. Therefore, $\alpha(y)$ is in the boundary of $\m K^{\bar{R}}$. This implies that $y = sy_1 + \partial(y_2)$ where $y_1 \in K^R$, $y_2 \in \m K^R$  and $\soc(R) = (s)$. 

We know by \cite[Lemma 1.2]{lev1} that 
$\soc(R) K^R_i \subset (0 : \m^2)\Bo_{i}(K^R)$ for $1 \leq i \leq n - 1$. 
If $\deg(y) = \deg(y_1) < n$, then $sy_1 \in (0 : \m^2)\Bo(K^R)$ and consequently $y \in \m \Bo(K^R) \subset \m \Bo(U)$. On the other hand if $\deg(y) = n$, then $y_2 = 0$ and $y = s y_1 = a s X_1\ldots X_n$, $a \in R$. Since $\Ho(K^R)$ is a Poincar\'e algebra \cite{av5}, there is a $z' \in \Zi_{n-1} (K^R)$ such that $\bar{z}z' = s X_1\ldots X_n$. We conclude $y = a\bar{z}z'= \partial (aTz') \in \m \Bo(U)$. Therefore, the induction step for $m = 0$ follows.

Now we assume that $m > 0$. Note that $a_0 = v^m (\i(y))$. Since $\i(y) \in \m \Bo(\mathfrak{X})$ and the chain derivation $v$ commutes with the differential of $\mathfrak{X}$, we have $a_0 \in \m K^R \cap \Bo(\m \mathfrak{X})$. We consider two cases. 

Suppose $\deg(a_0) = n$. Then $a_0 \in \Zi_n(\m K^R) = \soc(R)K^R_n$. Therefore, $a_0 = a s X_1\ldots X_n$, $a \in R$. One observes $a_0T^{(m)} = a\bar{z}z'T^{(m)} = \partial(az' T^{(m + 1)}) \in \m \Bo(U)$. 
Therefore, $ \i(\sum_{k=0}^{m-1} a_k T^{(m-k)}) = \i(y) - a_0T^{(m)}$ is in the boundary of $\m \mathfrak{X}$. Consequently $\sum_{k=0}^{m-1} a_k T^{(m-k)}$ is in the boundary of $\m U$ by the induction hypothesis. We conclude that $y$ is in the boundary of $\m U$. 

Suppose $\deg(a_0) < n$. Then by the argument in the induction step $m = 0$, one has $a_0 = \partial(y_3)$ for $y_3 \in \m K^R$. We can write 
\[y = \partial(y_3) T^{(m)} + \sum_{k=0}^{m-1} a_k T^{(m-k)} = \partial(y_3 T^{(m)}) + [(-1)^{\deg(y_3) + 1} y_3 \bar{z} T^{(m-1)} + \sum_{k=0}^{m-1} a_k T^{(m-k)}].\] 
The first summand is in $\m \Bo(U)$. This implies that the second summand is in the boundary of $\m \mathfrak{X}$ and therefore also in the boundary of $\m U$ by induction hypothesis. We conclude that $y$ is in the boundary of $\m U$. This completes the induction step. Hence $\i_*$ is an injective map and the statement (1) follows.

The ring $\bar{R}$ is Golod.  The Poincar\'e series of $R$ is computed as $\displaystyle P^{R}_k(t) = \frac{(1 + t)^n}{1 - t(P^Q_R(t) - 1) + t^{n+1}(1+t) }$ in \cite[Proposition 6.2]{rossi}, so the statement (2) follows from  Theorem \ref{prl9.5}.
\end{proof}

%%%%%%%%%%%%%%%%%%%%%%%%%%%%%%%%%%%%%%%%%%%%%%%%%%%%%%%%%%%%%%%%% 
\section{Stretched and almost stretched rings}

Our aim in this section is to prove that stretched and almost stretched Gorenstein rings are good. The key step is to show that rings of these type decompose as connected sums.  If the residue field is infinite, then  Lemma \ref{acg3} below follows from \cite[Theorem 1]{eakin}.  The proof of the lemma is suggested by the anonymous referee.

\begin{lemma}\label{acg3}
Let $(R, \m, k)$ be a local ring such that $\mu(\m^2) \leq 2$. Then there exists an $x \in \m \setminus \m^2$ such that $\m^2 = x\m$. Furthermore, if $\lo(R) \geq 3$, then $x^2 \not\in \m^3$. 
\end{lemma}
\begin{proof}
If $\m^2 = x \m + \m^3, x \in \m \setminus \m^2$, then by Nakayama's lemma we have $m^2 = x \m$. Therefore, to prove the first assertion, it is enough to assume that $\m^3 = 0$, i.e. $\m^2$ is a $k$-vector space. If $\m^2 = 0$, then $\m^2 = x \m = 0$ for all $x \in \m \setminus \m^2$. If $\mu(\m^2) = 1$, then for any $x \in \m \setminus \m^2$ such that $x \m \not= 0$, we have $\m^2 = x \m$. Therefore, we only need to consider the case when $\mu(\m^2) = 2$, i.e. $\m^2$ is a vector space of dimension two.

Let $x_1, \ldots, x_n$ be a minimal generating set of $\m$. Let $r$ be such that that $x_i \m \neq 0$ for all $i$
with $1\leq  i \leq r$ but $x_i \m = 0$ for $i > r$. Assume, by way of contradiction, that $\m^2 \neq x\m$ for all $x \in \m \setminus \m^2$. Thus, if $i \leq r$ then
$x_i\m$ is a one-dimensional vector space. We may also assume $\m^2 = x_1\m + x_2 \m$. Clearly  $x_1 \m \neq x_2 \m$ since otherwise $\m^2 = x_1 \m$, a contradiction to our assumption. 

If $i \leq r$, $j \leq r$ and $x_ix_j \neq 0$, then one observes that $x_i \m = x_j \m$.  This is true because $x_i\m$ and $x_j \m$ are both one-dimensional vector spaces and they share the nonzero element $x_i x_j$.
Since $x_1 \m \neq x_2 \m$, we must have thus $x_1x_2 = 0$. Now $x_1 \m \neq 0$ and $x_2 \m \neq 0$, so there exist $i, j$ with $i \leq r$, $j \leq r$ such that $x_1x_i \neq 0$ and $x_2 x_j \neq 0$. This implies $x_1\m = x_i \m$ and $x_2 \m = x_j \m$. In particular, $x_i \m \neq x_j \m$ and hence $x_ix_j = 0$.
We have $(x_2 + x_i)x_1 = x_1x_i \neq 0$ and $(x_2 + x_i)x_j = x_2x_j \neq 0$, hence $(x_2 + x_i) \m = x_1\m$ and $(x_2 + x_i) \m = x_j \m = x_2\m$. This yields $x_1 \m = x_2\m$, a contradiction. Therefore, the first part of the lemma follows.

If $x^2 \in \m^3$, then $\m^3 = x^2 \m \subset \m^4$. By Nakayama's lemma, $\m^3 = 0$ which implies $\lo(R) \leq 2$. Therefore, $x^2 \not\in \m^3$ if $\lo(R) \geq 3$.
\end{proof}

We recall definitions of fibre products and connected sums, cf. \cite{an3}.

\begin{definition}\label{prl10}
Let $(R, \m_R, k)$, $(S, \m_S, k)$ be local rings with a common residue field $k$. Let $\pi_R : R \rightarrow k$, $\pi_S : S \rightarrow k$ be natural quotient maps from $R$, $S$ onto $k$ respectively. The fibre product of $R$ and $S$ is defined as the ring $R \times_k S = \{(r, s) \in R \times S : \pi_R(r) = \pi_S(s) \}$. The ring $R \times_k S$ is local with maximal ideal $\m_R \oplus \m_S$. 
%The rings $R$, $S$ are modules over $R \times_k S$ by left actions defined by projection maps $(r, s) \mapsto r$, $(r, s) \mapsto s$ respectively. %For any $R$-module $M$ we have natural quotients $
      
Now assume that both $R$, $S$ are Artinian Gorenstein local rings with one dimensional socles $\soc(R) = \langle \delta_R \rangle$, $\soc(S) = \langle \delta_S \rangle$. Then the connected sum of $R$ and $S$ is defined as 
$$R \#  S = \frac{R \times_k S}{ \langle (\delta_R, - \delta_S)\rangle}.$$

We say a Gorenstein local ring $Q$ is decomposable as a connected sum if there are rings $R$, $S$ such that $Q = R \#  S$, $l(R) < l(Q)$ and $l(S) < l(Q)$. Here $l(-)$ denotes the length function. 
\end{definition}

 Define a left module structure on the polynomial ring $T = k[Y_1, \ldots, Y_n]$ over the ring $S = k[X_1, \ldots, X_n], X_i = Y_i^{-1}$ by defining the action of $X_i$ on a monomial $M \in T$ as the usual multiplication if $X_iM \in T$ and zero otherwise. Macaulay's inverse system establishes a one-to-one correspondence between local Artinian Gorenstein algebras $\displaystyle R = \frac{k[X_1, \ldots, X_n]}{I}$ such that $I \subset (X_1, \ldots, X_n)$ and polynomials $F$ in the ring $T$ up to a unit multiple. The correspondence is given by $I = \ann F$, cf.\cite[Theorem 21.6]{ebud}. If Gorenstein local rings $R$, $S$ correspond $F \in k[Y_1, \ldots, Y_m]$, $G \in k[Y_{m + 1}, \ldots, Y_n]$ respectively, then the connected sum $R \# S$ corresponds $F + G \in k[Y_1, \ldots, Y_n]$, cf. \cite[Remark 4.24]{an0}.

%Macaulay's inverse system establishes a one-to-one correspondence between local Artinian Gorenstein algebras $\displaystyle R = \frac{k[X_1, \ldots, X_n]}{I}$ such that $I \subset (X_1, \ldots, X_n)$ and polynomials $F$ in the ring $T = k[Y_1, \ldots, Y_n], Y_i = X_i^{-1}$ upto a unit multiple. The correspondence is given by $I = \ann F$, where the action of $X_i$ on a monomial $u \in k[Y_1, \ldots, Y_n]$ is the usual multiplication if $X_iu \in T$ and zero otherwise, cf.\cite[Theorem 21.6]{ebud}. If Gorenstein local rings $R$, $S$ correspond $F \in k[Y_1, \ldots, Y_m]$, $G \in k[Y_{m + 1}, \ldots, Y_n]$ respectively, then the connected sum $R \# S$ corresponds $F + G \in k[Y_1, \ldots, Y_n]$.
%

The following result follows from \cite[Proposition 4.1]{an1} and can be proved easily for Artinian $k$-algebras using Macaulay's inverse system.
\begin{theorem}\label{acg8}
Let $(Q, \n, k)$ be a regular local ring and $I \subset \n^2$ be an ideal such that $R = Q/ I$ is an Artinian Gorenstein local ring. Let $\n$ be minimally generated by $x_1, \ldots, x_m, y_1, \ldots, y_n$ such that $(x_1, \ldots, x_m)(y_1, \ldots, y_n)  \subset I$. Let $\max\{ i : (x_1, \ldots, x_m)^i \not\subset I\} = s$, $\max\{ i : (y_1, \ldots, y_n)^i \not\subset I\} = t$. Then there are ideals $I_1$, $I_2$ in $Q$ containing $(x_1, \ldots, x_m)$, $(y_1, \ldots, y_n)$ respectively such that the following hold.

\begin{enumerate}
\item
The rings $S = Q/ I_1$, $T = Q/ I_2$ are Gorenstein rings. $\edim(S) = n$, $\edim(T) = m$, $\lo(S) = t$, $\lo(T) = s$.
\item
$R = S \#_k T$.
\end{enumerate}
\end{theorem}

%The notion of a Conca generator is introduced in \cite{av7} by Avramov et al. A Conca generator of a local ring $(R, \m)$ is an element $x \in \m \setminus \m^2$ such that $\m^2 = x \m$ and $x^2 = 0$. We consider a weaker version below. 
%
%\begin{definition}{\rm (Almost Conca Generators)}\label{acg4}
%Let $(R, \m)$ be a local ring. An element $x \in \m \setminus \m^2$ is called an almost Conca generator of $\m$ if $\m^2 = x \m$.
%\end{definition}

The next theorem is the key to decompose an Artinian Gorenstein local ring $(R, \m, k)$ with $\mu(\m^2) \leq 2$ as a connected sum. When $\edim(R) = 2$ and $\Ch(k) = 0$, the theorem  follows from \cite[Theorem 4.1]{elias}.

%The following result is motivated by \cite[Theorem 4.1]{elias}.

\begin{theorem}\label{acg5}
Let $(R, \m, k)$ be a local Artinian Gorenstein ring. Let $\edim(R) = n$, $\lo(R) \geq 3$ and $\dim_k \m^2/ \m^3 = m < n$. 
Assume that $\m$ admits a generator $x_1$ such that $\m^2 = x_1 \m$. Then there exists  a minimal generating set $\{x_1, x_2, \ldots, x_n\}$ of $\m$ extending $x_1$ such that the following hold.

\begin{enumerate}
\item
$\m^2 = (x_1^2, x_1x_2, \ldots, x_1x_m)$, 
 
\item
$(x_1, \ldots, x_m)(x_{m +1}, \ldots, x_n) = 0$,

\item
$(x_{m +1}, \ldots, x_n)^2 = \soc(R)$.
\end{enumerate}
The ring $R$ decomposes as connected sum $R = S \# T$ such that $\edim(S) = m$, $\edim(T) = n - m$, $\lo(S) = \lo(R)$ and $\lo(T) = 2$.

\end{theorem}
\begin{proof}

Since $\lo(R) \geq 3$, we have $x_1^2 \not\in \m^3$. Therefore, we can choose a minimal generating set $\{x_1, x_2, \ldots, x_n\}$ of $\m$ such that $\m^2 = (x_1^2, x_1x_2, \ldots, x_1x_m)$. The statement (1) follows.

We have $x_1x_j = \alpha_{1j}x_1^2 + \alpha_{2j} x_1x_2 + \ldots + \alpha_{mj}x_1x_m$, $\alpha_{ij} \in R$ for $1 \leq i \leq m$ and $m +1 \leq j \leq n$. This gives $x_1(x_j - \alpha_{1j}x_1 - \alpha_{2j}x_2 - \ldots - \alpha_{mj}x_m) = 0$. Replacing $x_j - \alpha_{1j}x_1 - \alpha_{2j}x_2 - \ldots - \alpha_{mj}x_m$ by $x_j$, we assume that $x_1x_j = 0$ for $m + 1 \leq j \leq n$. 

If $m = 1$, the property (2) is satisfied. We assume that $\mu(\m^2) = m \geq 2$.
Since $\m^2$ is minimally generated by $\{x_1^2, x_1x_2, \ldots, x_1x_m\}$ and $x_1(x_{m+1}, \ldots, x_n) = 0$, we have 
\[(0 : x_1) \subset \m (x_1, \ldots, x_m) + (x_{m+1}, \ldots, x_n).\] 
This implies that residue classes of elements $x_{m + 1}, \ldots, x_n$ form a $k$-basis of $\displaystyle \frac{(0 : _R x_1) + \m^2}{\m^2}$. Therefore, $\displaystyle \dim_k \frac{(0 : _R x_1) + \m^2}{\m^2} = n - m$.

\noindent
{\bf Claim - 1} : $\displaystyle \m[(0 : _R x_1) \cap \m^2] = \soc(R)$.

Note that $\m^2(0 : _R x_1) = x_1\m (0 : _R x_1) = 0$. This implies that $\m[(0 : _R x_1) \cap \m^2] \subset \soc(R)$. Therefore, it is enough to prove that $\m[(0 : _R x_1) \cap \m^2] \not= 0$. We have $\m^{i +1} = x_1\m^i$ for $i \geq 1$. This means that $\mu(\m^{i +1}) \leq \mu(\m^i), i \geq 1$. Let $t = \mx\{i : \mu(\m^i) = m\}$. Then $t \geq 2$. Since $m \geq 2$, we have $\lo(R) \geq t + 1$. The map $\displaystyle \m^t/ \m^{t + 1} \xrightarrow{\cdot x_1} \m^{t + 1}/ \m^{t + 2}$ is not injective since $\dim_k(\m^t/ \m^{t + 1}) >  \dim_k(\m^{t +1}/ \m^{t + 2})$. Therefore, we find  $y \in \m^t \setminus \m^{t +1}$ such that $yx_1 \in \m^{t + 2}$. Note that $\m^{t+2} = x_1 \m^{t +1}$. It follows that $yx_1 = x_1m$ for some $m \in \m^{t +1}$. Consequently $x_1(y - m) = 0$.  Clearly, $y - m  \in [(0 : _R x_1) \cap \m^2]$ and $y - m \not\in \m^{t +1}$. Since $\soc(R) \subset \m^{t + 1}$, we have $y - m \not\in \soc(R)$. Therefore, $[(0 : _R x_1) \cap \m^2] \not \subset \soc(R)$. We conclude that $\m[(0 : _R x_1) \cap \m^2] \not=0$ and the claim is proved.

\noindent
{\bf Claim - 2} : $ \displaystyle \dim_k \frac{[(0 : _R x_1) \cap \m^2]}{\m[(0 : _R x_1) \cap \m^2]} = m - 1$.

We know that $(0 : _R x_1) = \Hom_R(R/ (x_1), R)$. We have $\m^2 = x_1\m \subset (x_1)$. By Matlis duality $\displaystyle l(0 : _R x_1) = l(R/ x_1R) = l(R/ \m^2) - l(\frac{x_1R + \m^2}{\m^2}) = 1 + n - 1 = n$. Now we have
\[
l[(0 : _R x_1) \cap \m^2]  =  l(0 :_R x_1) + l(\m^2) - l[(0 : _R x_1) + \m^2]
 =  l(0 :_R x_1) - l [\frac{(0 : _R x_1) + \m^2}{\m^2}]
 =  n - (n - m) = m.
\]
Therefore, $ \displaystyle \dim_k \frac{[(0 : _R x_1) \cap \m^2]}{\m[(0 : _R x_1) \cap \m^2]} = l(\frac{[(0 : _R x_1) \cap \m^2]}{\m[(0 : _R x_1) \cap \m^2]}) = l[(0 : _R x_1) \cap \m^2] - 1 = m - 1$. 

\noindent
{\bf Claim - 3} : The pairing $\displaystyle \frac{(x_2, \ldots, x_m)}{\m(x_2, \ldots, x_m)} \times \frac{[(0 : _R x_1) \cap \m^2]}{\m[(0 : _R x_1) \cap \m^2]} \rightarrow \soc(R)$ given by $(\bar{x}, \bar{y}) \rightarrow xy$ is well defined and non-degenerate.

We have $\m (x_2, \ldots, x_m)(0 : _R x_1) = 0$ since $\m^2 = x_1\m$. This implies that ${(x_2, \ldots, x_m)[(0 : _R x_1) \cap \m^2]}  \\ \subset \soc(R)$. Therefore, the above pairing exists. Note that $ \m^2( x_{m + 1}, \ldots, x_n) = 0$. As a result if $y \in [(0 : _R x_1) \cap \m^2]$ and $y(x_2, \ldots, x_m) = 0$, we have $y \in \soc(R) = \m[(0 : _R x_1) \cap \m^2]$. 
This implies that the map  
\[\frac{[(0 : _R x_1) \cap \m^2]}{\m[(0 : _R x_1) \cap \m^2]} \rightarrow \Hom\Big(\frac{(x_2, \ldots, x_m)}{\m(x_2, \ldots, x_m)}, \soc(R)\Big)\]
induced by the above pairing is injective.
We have $\displaystyle\dim_k \frac{(x_2, \ldots, x_m)}{\m(x_2, \ldots, x_m)} = \dim_k \frac{[(0 : _R x_1) \cap \m^2]}{\m[(0 : _R x_1) \cap \m^2]} = m - 1$. Therefore, the above map is an isomorphism and consequently the pairing is non-degenerate. 

Note that $\m(x_2, \ldots, x_m)(x_{m +1}, \ldots, x_n) = 0$ so $(x_2, \ldots, x_m)(x_{m +1}, \ldots, x_n) \subset \soc(R)$. The multiplication by $x_j$ defines a map $\displaystyle \frac{(x_2, \ldots, x_m)}{\m(x_2, \ldots, x_m)} \rightarrow \soc(R)$ for $m+1 \leq j \leq n$. Since the pairing in Claim - 3 is non-degenerate, we have $\hat{x}_j \in [(0 : _R x_1) \cap \m^2]$ such that $x_jx_i = \hat{x}_jx_i$ for $2 \leq i \leq m$ and $m + 1\leq j \leq n$. We also have $x_1x_j = x_1\hat{x}_j = 0$ for $m+1 \leq j \leq n$.
It follows that $(x_j - \hat{x}_j)(x_1, \ldots, x_m) = 0$, $m + 1\leq j \leq n$. Therefore, replacing $(x_j - \hat{x}_j)$ by $x_j$ we have $(x_1, \ldots, x_m)(x_{m +1}, \ldots, x_n) = 0$ and the property (2) is satisfied.

Let $\soc(R) = \langle\delta\rangle$ and $K = (x_{m +1}, \ldots, x_n)$. We have $\m^2K = x_1K\m = 0$, so  $K^2 \subset \m K \subset \soc(R)$. Note that $K^2 \not= 0$ for otherwise each of $x_j, m + 1 \leq j \leq n$ is in $\soc(R) \subset \m^2$, a contradiction. Therefore, $K^2 = \m K = \soc(R)$ and the property (3) is satisfied. 

The last statement follows from Theorem \ref{acg8}.
\end{proof}

The following is a consequence of Lemma \ref{acg3}, Theorem \ref{acg5}.
\begin{corollary}\label{ascs}
Let $(R, \m)$ be an Artinian Gorenstein local ring such that $\lo(R) \geq 3$ and $\mu(\m^2) \leq \min \{2, \edim(R) -1\}$. Then $R = S \# T$ where $(S, \p)$, $(T, \q)$ are Gorenstein local rings, $\edim(S) = \mu(\m^2)$,  $\lo(S) = \lo(R)$ and $\lo(T) = 2$.
\end{corollary}

\begin{lemma}\label{ds3}
Let $(R, \m)$ be an Artinian Gorenstein local ring such that $\edim(R) \geq 2$ and $\lo(R) = s$. Then the quotient ring $R/ \m^i$ is not a Gorenstein ring for $2 \leq i \leq s$.
\end{lemma}

\begin{proof}
If possible assume that $R/ \m^i$ is a Gorenstein ring for some $i$ satisfying $2 \leq i \leq s$. Then the injective hull of $k$ over $R/ \m^i$ is $E_{R/ \m^i}(k) \cong R/ \m^i$. We know that $E_{R/ \m^i}(k) \cong \Hom_R(R/ \m^i, R) = (0 :_R \m^i)$. Consequently  $(0 :_R \m^i) = (x)$, a principal ideal for some $x \in R$. Now $\m^{s - i +1} \subset (0 :_R \m^i)$ and $\m^{s - i +1} \not \subset \m (0 :_R \m^i)$ for otherwise $\m^{s - i +1} \subset \m (0 :_R \m^i) \subset (0 :_R \m^{i-1})$ which implies that $\m^s = \m^{s-i+1} \m^{i-1} = 0$, a contradiction. Since $(0 :_R \m^i)$ is principal, we have $(0 :_R \m^i) = \m^{s-i+1} = (x)$.

 Apply  Macaulay's theorem characterising  Hilbert function \cite[Theorem 4.2.10]{bruns} to the associated graded ring $\gr_{\m}(R)$. We obtain $\mu(\m^{n+1}) \leq \mu(\m^n)^{\langle n \rangle}$ for all $n \geq 1$. We already have $\mu(\m^{s-i+1}) = 1$.
This implies that $\m^j$ is a principal ideal for $j$ satisfying the inequality $s-i+1 \leq j \leq s$ and so $l(0 :_R \m^i) = l( \m^{s-i+1}) = i$. By Matlis duality, $l(R/ \m^i) = l \Hom(R/ \m^i, R) = l(0 :_R \m^i) = i$. We have  $\displaystyle \sum_{j = 0}^{i-1}[l(\m^j/ \m^{j+1}) - 1] = l(R/ \m^i) - i = 0$ and each summand is non-negative. This shows that $\edim(R) = l(\m/\m^2) = 1$, a contradiction. %Therefore, the lemma follows.
\end{proof}

The following theorem is proved by Dress and Kr\"amer in \cite[Satz 2]{dress}. 
\begin{theorem}\label{prl10.6}
Let $(S, \m_S, k)$, $(T, \m_T, k)$ be two local rings and $R = S \times_k T$. Then 
$$\frac{1}{P^R_k(t)} = \frac{1}{P^S_k(t)}  + \frac{1}{P^T_k(t)} - 1.$$ If $M$ is an $S$-module, then 
\[\frac{1}{P^R_M(t)} = \frac{P^S_k(t)}{P^S_M(t)}\Big(\frac{1}{P^S_k(t)}  + \frac{1}{P^T_k(t)} - 1\Big) = \frac{P^S_k(t)}{P^S_M(t)P^R_k(t)}.\]
\end{theorem}

The following is a consequence of Lemma \ref{acg3} and Theorem \ref{acg5}.
\begin{lemma}\label{ds4}
Let $(R, \m, k)$ be an Artinian Gorenstein local ring such that $\mu(\m) = n$, $\lo(R) \geq 2$ and $\mu(\m^2) \leq 2$. Let $i$ be an integer satisfying $2 \leq i \leq \lo(R)$. Then $R/ \m^i$ is a Golod ring  and $ \displaystyle P^{R/ \m^i}_k(t) = \frac{1}{1 - nt}$. 
\end{lemma}

\begin{proof}
Fix $i$ satisfying $2\leq i \leq \lo(R)$ and set $\displaystyle \bar{R} = \frac{R}{\m^i}$. The result is clear when $n=1$ or $\lo(R) = 2$. First we assume that $n = 2$ and $\lo(R) > 2$. A result of Scheja \cite{scheja} states that a codepth 2 local ring is either a Gorenstein (equivalently complete intersection) or a Golod ring. The ring $\bar{R}$ cannot be a Gorenstein ring by Lemma \ref{ds3} so $\bar{R}$ is a Golod ring. Since $R$ is a complete intersection, the defining ideal of $\bar{R}$ is minimally generated by $3$ elements. It follows that, $\kappa^{\bar{R}}(t) = \sum_{i \geq 0} \dim_k \Ho_i(K^{\bar{R}})t^i = 1 + 3t + 2t^2$ and

\[(1 - t(\kappa^{\bar{R}}(t) - 1)) = 1 - t(1 + 3t + 2t^2 - 1) = (1 + t)^2(1 - 2t).\] We have $\displaystyle P^{\bar{R}}_k(t) = \frac{(1+t)^2}{1 - t(\kappa^{\bar{R}}(t) - 1)} = \frac{1}{1 - 2t}$. 

Now we assume that $n > 2$ and $\lo(R) > 2$. By Corollary \ref{ascs}, it follows that $R = S \# T$ where $(S, \p)$, $(T, \q)$ are Gorenstein local rings, $\edim(S) = \mu(\m^2) \leq 2$, 
%$\edim(T) + \edim(S)$ = n
, $\lo(S) = \lo(R) \geq 3$ and $\lo(T) = 2$. One has $\bar{R} = \bar{S} \times_k \bar{T}$ where $\bar{S} = S/ \p^i$ and $\bar{T} = T/ \q^2$.
Both $\bar{S}$, $\bar{T}$ are Golod rings. The ring $\bar{R}$ is a Golod ring because a fibre product of Golod rings is Golod , cf.\cite[Theorem 4.1]{les1}.

We have $\displaystyle P^{\bar{S}}_k(t) = \frac{1}{1- \edim(\bar{S})t}$ by the case $n = 2$ and $\displaystyle P^{\bar{T}}_k(t) = \frac{1}{1-\edim(\bar{T})t}$ (\cite[Example 4.2.2]{av}). The formula for the Poincar\'e series follows from Theorem \ref{prl10.6}.
\end{proof}

The following is a well known fact.
\begin{lemma}\label{ds4.5}
Let $(R, \m, k)$ be a local ring, $x$ be a nonzero divisor of $R$ and $S = R/ (x)$. If there exists a polynomial $d(t) \in \Zi[t]$ such that $d(t) P^S_M(t) \in \Z[t]$ for all $S$-module $M$, then $d(t) P^R_M(t) \in \Z[t]$ for all $R$-module $M$. 
Now assume further that $x \in \m \setminus \m^2$, then $\displaystyle P^S_k(t) = \frac{P^R_k(t)}{1 + t}$. The ring  $S$ is Golod if and only if $R$ is so.
\end{lemma} 

\begin{proof}
Let $M$ be an $R$-module and $N$ be the first syzygy of $M$. We have the following exact sequence: 
\[0 \rightarrow N \rightarrow R^{\mu(M)} \rightarrow M \rightarrow 0.\]

This implies that $P^R_M(t) = \mu(M) + t P^R_N(t)$. Therefore, it is enough to show that $d(t) P^R_N(t) \in \Z[t]$. Let $F_* \twoheadrightarrow N$ be the minimal free resolution of $N$ over $R$. Note that $x$ is also a nonzero divisor of $N$. This implies that $S \otimes_R F_* \twoheadrightarrow S \otimes_R N$ is a minimal free resolution of $S \otimes_R N = \bar{N}$ as an $S$ module. As a result, we have $P^R_N(t) = P^S_{\bar{N}}(t)$. Therefore, the first part of the lemma follows from the hypothesis. 

The assertions regarding Poincar\'e series and Golod property follow from  \cite[Proposition 3.3.5.(1)]{av}, \cite[Proposition 5.2.4]{av} respectively.
\end{proof}

The following is due to \c Sega \cite[Proposition 1.5]{sega1}.

\begin{proposition}\label{ar1}
Let $R$ be a local ring such that there is a $d_R(t) \in \Z[t]$ satisfying $d_R(t) P^R_M(t) \in \Z[t]$ for each finitely generated $R$-module $M$. Let $d_R(t) = p(t)q(t)r(t)$ where $p(t)$ is $1$ or irreducible; $q(t)$ has non-negative coefficients; $r(t)$ is $1$ or irreducible and has no positive real root among its complex roots of minimal absolute value.
Then the following hold for each pair of $R$-modules $M$, $N$.
\begin{enumerate}
\item 
If $\Tor^R_i(M, N) = 0$ for $i \gg 0$, then either $M$ or $N$ has finite projective dimension. 

\item
If $\Ext^i_R(M, N) = 0$ for $i \gg 0$, then either $M$ has a finite projective dimension or $N$ has a finite injective dimension.
\end{enumerate}
\end{proposition}

We are prepared to prove Theorem \ref{intro3.5}.

\subsection{Proof of Theorem \ref{intro3.5}}
\begin{proof}
%By Lemma \ref{ds4.5}, it is enough to prove statements (1), (2) when the ring $R$ is Artinian, i.e. $d = 0$. 
The case $n = 1$ is easy. We skip the details.

%:
%We assume that $n = 1$. Then $P^{R}_k(t) = \frac{1}{1 - t}$. The ring $R$ is a quotient of a DVR. Using the structure theorem of modules over PID, the module $M$ is a direct summand of quotients of $R$. So $P^R_M(t) = \mu(M) + \mu(\syz_1^{R}(M))\frac{t}{1 - t}$. This implies that $(1- t)P^R_M(t) \in \Z[t]$. So (1) follows.

Now we assume that $n \geq 2$. The quotient map $\displaystyle R \twoheadrightarrow \frac{R}{\soc(R)}$ is a Golod homomorphism and $\displaystyle P^{R/ \soc(R)}_k(t) = \frac{P^R_k(t)}{ 1 - t^2 P^R_k(t)}$, cf. \cite[Theorem 2]{lev1}. 
Therefore, we have
\[
P^{R}_k(t) 
= \frac{P^{R/ \soc(R)}_k(t)}{ 1 + t^2 P^{R/ \soc(R)}_k(t)}
= \frac{1}{ \frac{1}{P^{R/ \soc(R)}_k(t)} + t^2}
=\frac{1}{1 - nt + t^2}.
\]
The last equality follows from Lemma \ref{ds4}. The same Lemma states that $R/ \soc(R)$ is a Golod ring. By Theorem \ref{intro2}, the ring $R$ is a surjective image of a complete intersection under a Golod homomorphism. The second part of the  (2) is a consequence of Theorem \ref{prl9.5}.

Now we prove (3). If the projective dimension $\pd_R(M)$ of $M$ is finite, then $\Ext^{\pd_R(M)}_R(M, M) \neq 0$. Therefore, it is enough to show that $\pd_R(M) < \infty$.
When $n \leq 2$, $R$ is a complete intersection. The statement follows from a result of Avramov and Buchweitz, cf. \cite[Theorem 4.2]{avbu}. When $n \geq 3$, the polynomial $(1 - nt + t^2)$ is irreducible. The statement follows from Proposition \ref{ar1} (take $p(t) = (1 - nt + t^2)$, $q(t) = (1 + t)^{n}$ and $r(t) = 1$). Here one uses the fact that a module over a Gorenstein ring has a finite projective dimension if and only if it has a finite injective dimension. 
%a result of Sega \cite[Proposition 1.5]{sega1} and by an argument similar to \cite[Corollary 1.8]{sega1}.
\end{proof}

Stretched and almost stretched rings were introduced by Sally \cite{sally} and Elias, Valla \cite{elias} respectively. 

\begin{definition}\label{intro0}
An Artinian local ring $R$ with maximal ideal $\m$  is called stretched if $\m^2$ is a principal ideal and  almost stretched if $\m^2$ is minimally generated by two elements.

Let $R$ be a Cohen-Macaulay local ring of dimension $d$ with maximal ideal $\m$.  Then $R$ is called stretched (almost stretched) if there exists a minimal reduction $\underline{x} = x_1, \ldots, x_d$ of $\m$ 
%(i.e. there exist $d$ elements $x_1, \ldots, x_d$ of $\m$ such that $\m^{r+1} = (x_1, \ldots, x_d)\m^r$ for some non-negative integer $r$)
 such that $R/(\underline{x})$ is a stretched (almost stretched) Artinian ring. Here by minimal reduction we mean that $\underline{x}$ satisfies $\m^{r+1} = (x_1, \ldots, x_d)\m^r$ for some non-negative integer $r$.
\end{definition}

Stretched Cohen-Macaulay local rings are shown to be good in a recent article \cite[Corollary 5.6]{amenda}. We outline a different method. 
In statement (2) of the following corollary, we find a more efficient common denominator of Poincar\'e series of modules over such rings.

\begin{corollary}\label{sbr4}
Let $(R, \m, k)$ be a $d$-dimensional stretched Cohen-Macaulay local ring and $M$ be an $R$-module. Let $n = \dim_k \m/ \m^2$ and $r = \dim_k \Ext_R^d(k, R)$ denote the type of $R$. Then the following hold. 
\begin{enumerate}
\item
If $r = n - d$, then $R$ is a Golod ring, $\displaystyle P^R_k(t) = \frac{(1+t)^d}{1 - (n-d)t}$ and $(1 - (n-d)t) P^R_M(t) \in \Z[t]$.

\item
If $r \neq n - d$, then $\displaystyle P^R_k(t) = \frac{(1+t)^d}{1 - (n-d)t + t^2}$ and $(1 + t)^{n-d - r + 1}(1 - (n-d)t + t^2) P^R_M(t) \in \Z[t]$.

\item
If $\Ext^i(M, M \oplus R) = 0$ for all $i \geq 1$, then $M$ is a free $R$-module.
\end{enumerate}
\end{corollary}

\begin{proof}
To prove statements (1), (2), it is enough to assume that $R$ is a stretched Artinian ring, i.e. $d = 0$ (see Lemma \ref{ds4.5}). 
We have $\edim(R) = n$ and  $\dim_k \soc(R) = r$. Let $\m = (x_1, \ldots, x_n)$ and $\lo(R) = s$. The ideals $\m^i, i \geq 2$ are principal ideals. If $\soc(R) \subset \m^2$, then $\soc(R) = \m^s$ is a principal ideal, so $R$ is a Gorenstein ring. Both statements (1), (2) follow from Theorem \ref{intro3.5}.

Otherwise assume that $x_1 \in \soc(R) \setminus \m^2$. One observes that $(x_1) \cap (x_2, \ldots, x_n) = 0$ and $(x_1) + (x_2, \ldots, x_n) = \m$. For any two ideals $I, J$ in a ring $R$, we know that $R = R/ I \times_{R/ I + J} R/ J$. Therefore, it follows that  $R = R/(x_1) \times_k R/(x_2, \ldots, x_n)$. 

The maximal ideal of the ring $R/(x_2, \ldots, x_n)$ is generated by the residue class of $x_1$, so its square is zero since $x_1 \in \soc(R)$. On the other hand, the ring $R/(x_1)$ is a stretched Artinian ring. If the socle of $R/(x_1)$ is contained in the square of its maximal ideal, it is a Gorenstein ring. Otherwise we decompose $R/(x_1)$ again as before.

After a finite number of steps, we have $R = S \times_k T$ where $(S, \m_S)$ is a stretched Artinian Gorenstein ring and $(T, \m_T)$ is a local ring with $\m_T^2 = 0$. Clearly $r = 1 + \edim(T)$. This implies that $\edim(T) = r - 1$ and $\edim(S) = \edim(R) - \edim(T) = n - r + 1$. 
By Theorem \ref{intro3.5}, we have 
%$P^S_k(t) = \frac{1}{1 - t(n - r + 1) + t^2}$ when $n \geq r + 1$ and $P^S_k(t) = \frac{1}{1 - t}$ when $n = r$ . 

\[
P^S_k(t) =  
\begin{cases}
\frac{1}{1 - t(n - r + 1) + t^2} \ &\text{when} \ n \geq r + 1,\\
\frac{1}{1 - t} \ &\text{when} \ n = r.
\end{cases}
\]
On the other hand $T$ is a Golod ring and $\displaystyle P^T_k(t) = \frac{1}{1 - (r-1)t}$.
The rational expression of $P^R_k(t)$ follows by the following computation: 
\begin{align*}
\frac{1}{P^R_k(t)} = \frac{1}{P^S_k(t)}  + \frac{1}{P^T_k(t)} - 1  &= 
\begin{cases}
(1 - t) + 1- (r-1)t - 1 \ & \text{when} \ n = r\\
(1 - t(n - r + 1) + t^2)  + 1- (r-1)t - 1 \ & \text{when} \ n \geq r + 1
\end{cases}\\
& =
\begin{cases}
1- nt \ & \text{when} \ n = r\\
1 - nt + t^2 \ & \text{when} \ n \geq r + 1.
\end{cases}
\end{align*}

If $n = r$, then $\edim(S) = 1$. This implies that $S$ is a Golod ring. Therefore, $R$ is also a Golod ring since a fibre product of Golod rings is Golod, cf. \cite[Theorem 4.1]{les1}.

Now we find a polynomial $d_R(t) \in \Z[t]$ such that $d_R(t)P^R_M(t) \in \Z[t]$ for any $R$-module $M$.
The second syzygy of $M$ is a direct sum of two modules, one is over $S$ and another over $T$, cf. \cite[Rem. 3]{dress}. Therefore, it suffices to assume that $M = M_1 \oplus M_2$ where $M_1$ and $M_2$ are modules over $S$ and $T$ respectively. 
By Theorem \ref{prl10.6} we have
\[
\frac{P^R_M(t)}{P^R_k(t)} = \frac{1}{P^R_k(t)} (P^R_{M_1}(t) + P^R_{M_2}(t))
= \frac{P^S_{M_1}(t)}{P^S_k(t)} + \frac{P^T_{M_2}(t)}{P^T_k(t)}.
\]

We observe that $\displaystyle \frac{P^T_{M_2}(t)}{P^T_k(t)} = \frac{1}{P^T_k(t)}(1 + t P^T_{\syz^T_1(M_2)}(t)) = 1 - (r-1)t + t \frac{P^T_{\syz^T_1(M_2)}(t)}{P^T_k(t)}$. Since the square of the maximal ideal of $T$ is zero, the first syzygy $\syz^T_1(M_2)$ is a $k$-vector space. Therefore, $\displaystyle \frac{P^T_{M_2}(t)}{P^T_k(t)}$ is a polynomial in $\Z[t]$.  

If $n = r$, we have $\edim(S) = 1$. 
%By Cohen's structure theorem, we have  $S = D/ (\pi^n)$ for some DVR $(D, (\pi))$. The module $M_1$ is a module over $D$, which is annihilated by $\pi^n$. By the structure theorem of modules over PID, it follows that $M_1$ is a direct sum of modules of the form  $D/ (\pi^m), 1 \leq m \leq n$ each of which has Poincar\' e series $\displaystyle \frac{1}{1 - t}$ over $S$. This implies that 
By (1) of Theorem \ref{intro3.5}, $\displaystyle \frac{P^S_{M_1}(t)}{P^S_k(t)}$  is a polynomial. Hence we conclude that $\displaystyle (1 - nt)P^R_M(t) = \frac{P^R_M(t)}{P^R_k(t)} \in \Z[t]$ if $n =r$.

If $n \geq r + 1$, we have $\edim(S) = n - r +1$. By (2) of Theorem \ref{intro3.5}, we have $\displaystyle (1 + t)^{(n - r+1)}\frac{P^S_{M_1}(t)}{P^S_k(t)} \in \Z[t]$. Hence we conclude that $\displaystyle (1 + t)^{(n - r+1)}(1 - nt + t^2)P^R_M(t) = (1 + t)^{(n - r+1)}\frac{P^R_M(t)}{P^R_k(t)} \in \Z[t]$ if $n \geq r + 1$.

Therefore, both  statements (1) and (2) follow.
To prove statement (3), it suffices to show that $\pd_R(M) < \infty$.
If $n - d = 2$, then $R$ is either a complete intersection or a Golod ring, cf. \cite{scheja}. In both cases, rings are known to satisfy statement (3), see for instance \cite[Theorem 4.2]{avbu} when $R$ is a complete intersection and \cite[Proposition 1.4]{jor} when $R$ is a Golod ring. If $n - d > 2$, then we see at once that $\pd_R(M) < \infty$ from Proposition \ref{ar1}. Here one observes that if injective dimension of $M \oplus R$ is finite then $R$ is Gorenstein and both projective and injective dimensions of $M$ are finite.
%The statement (3) can be proved by the same method as in its counterpart in Theorem \ref{intro3.5}. 
\end{proof}

The following result follows from Lemma \ref{ds4.5} and Theorem \ref{intro3.5}.
%\begin{corollary}\label{sbr4.6}
%Let $(R, \m, k)$ be a $d$-dimensional almost stretched Gorenstein local ring and $M$ be an $R$-module. Then $R$ satisfies assertions of Theorem \ref{intro3.5}. 
%\end{corollary}

\begin{corollary}\label{sbr4.6}
Let $(R, \m, k)$ be an almost stretched Gorenstein local ring of dimension $d$ and embedding dimension $n$. Let $M$ be a finitely generated $R$-module. 
%Let $\eta : Q \twoheadrightarrow R$ be a minimal Cohen presentation, $\n$ be the maximal ideal of $Q$ and $I = \ker \eta \subset \n^2$. 
Then the following hold.

\begin{enumerate}
\item
If $n - d = 1$, then $P^R_k(t) = \frac{(1 + t)^d}{1 - t}$ and $(1- t)P^R_M(t) \in \Z[t]$,

\item
If $n - d \geq 2$, then $P^R_k(t) = \frac{(1 + t)^d}{1 -(n - d)t + t^2}$ and $(1+t)^{n - d}(1 - (n - d)t + t^2)P^R_M(t) \in \Z[t]$,

\item
If $\Ext^i(M, M) = 0$ for all $i \geq 1$, then $M$ is a free $R$-module.
\end{enumerate}
\end{corollary}

%%%%%%%%%%%%%%%%%%%%%%%%%%%%%%%%%%%%%%%%%%%%%%%%%%%%%%%%%%%

\section{Revisiting known results}

In this section, we provide proofs of examples (1) and (3) in the introduction. As the title of the section suggests, these examples   were found by other authors. Our proofs are different and obtained using Theorem \ref{intro2}. Moreover our versions are slightly stronger, cf. Remark  \ref{rem}.
We recall the following which is proved in \cite[Theorem 1]{lev1}. 

\begin{theorem}\label{ggc1}
Let $(R, \m)$ be an Artinian Gorenstein local ring of embedding dimension $n$ and $K^R$ be the Koszul complex on a minimal set of generators of the maximal ideal $\m$. Set $\soc(R) = (s)$, $\bar{R} = R/ sR$ and $K^{\bar{R}} = \bar{R} \otimes_R K^R$, the Koszul complex of $\bar{R}$. Define a DG algebra structure on  $\displaystyle K^R \oplus \frac{K^R}{\m K^R}[-1]$ with multiplication : $(k, \bar{l})(k', \bar{l'}) = (kk', \overline{lk'} + (-1)^{\deg(k)}\overline{kl'})$ and differential : $\partial(k, \bar{l}) = (\partial(k) + sl, 0)$. Then the chain map $\displaystyle K^R \oplus \frac{K^R}{\m K^R}[-1] \rightarrow K^{\bar{R}}$, $(k, \bar{l}) \mapsto \bar{k}$ is a quasi-isomorphism. If $\displaystyle \bar{H} = \frac{\Ho(K^R)}{\Ho_n(K^R)}$, $\displaystyle \bar{K} = \frac{K^R \otimes_R k}{K_n^R \otimes_R k}[-1]$, then $\Ho(K^{\bar{R}}) = \bar{H} \lJoin \bar{K}$.
\end{theorem}

\begin{lemma}\label{sbr4.5}
Let $(R, \m)$ be an Artinian Gorenstein local ring of embedding dimension $n \geq 2$. Let $K^R$ denote the Koszul complex on a minimal set of generators of $\m$ and $C$ denote the quotient of $K^R$ defined by  
\[
C_i = 
\begin{cases}
K^R_i & \text{for} \ 0 \leq i \leq n - 2,\\
\frac{K^R_{n -1}}{\Bo_{n -1}(K^R)} & \text{for} \ i = n - 1,\\
0 & \text{for} \ i = n.
\end{cases}
\]
Then C has a DG algebra structure. 
Assume that $C$ is a Golod DG algebra with natural augmentation. Then  $R/ \soc(R)$ is a Golod ring and $R$ satisfies assertions (1), (2) of Theorem \ref{intro2}.
\end{lemma}

\begin{proof}
The fact that $C$ is a DG algebra is straightforward because $K^R_n \oplus \Bo_{n-1}(K^R)$ is a DG ideal of the Koszul algebra $K^R$. Let $q : K^R \twoheadrightarrow C$ be the quotient map and $\soc(R) = (s)$. 
Let $\displaystyle G : \frac{K^R}{\m K^R} \rightarrow K^R$, $\displaystyle H : \frac{C}{\m C} \rightarrow C$ denote chain maps induced by multiplications by $s$ on $K^R$, $C$ respectively. 
We have $q \circ G = H \circ \bar{q}$ where $\displaystyle \bar{q} : \frac{K^R}{\m K^R} \twoheadrightarrow \frac{C}{\m C}$ is the map induced by $q$. 
Let $\displaystyle \cn(G) = K^R \oplus \frac{K^R}{\m K^R}[-1]$, $\displaystyle \cn(H) = C \oplus \frac{C}{\m C}[-1]$
be the cones of $G$, $H$ respectively. 
Define $\alpha : \cn(G) \twoheadrightarrow \cn(H)$ by $\alpha(k, \bar{l}) = (q(k), \bar{q}(\bar{l}))$. Both  $\cn(G)$, $\cn(H)$ have DG algebra structure and $\alpha$ is a surjective DG algebra homomorphism. The kernel of $\alpha$ is the complex 
\[0 \rightarrow \frac{K^R_n}{\m K^R_n} \xrightarrow{\cdot s} K^R_n \xrightarrow{\partial} {\Bo}_{n-1}(K^R) \rightarrow 0\]
which is exact. Therefore, $\alpha$ is a quasi-isomorphism of DG algebras. By Theorem \ref{ggc1}, $\cn(G)$ is quasi-isomorphic to $K^{R/ \soc(R)}$. Therefore, to show that $R/ \soc(R)$ is a Golod ring it is enough to prove that $\cn(H)$ is a Golod algebra.

Since $C$ is a Golod algebra, there is a  $k$-basis $\mathfrak{b}_C = \{h_{\lambda}\}_{\lambda \in \Lambda}$ of $\Ho_{\geq 1}(C)$ and a function (trivial Massey operation) $\mu : \sqcup_{i = 1}^{\infty} \mathfrak{b}_C^i \rightarrow  C$ such that $\mu(h_\lambda) \in \Zi_{\geq 1}(C)$ with  $cls(\mu(h_\lambda))= h_\lambda$ for all $\lambda \in \Lambda$ and 
\[\partial \mu(h_{\lambda_1}, \ldots ,h_{\lambda_p}) = \sum_{j = 1}^{p-1}\overline{\mu(h_{\lambda_1}, \ldots ,h_{\lambda_j})}\mu(h_{\lambda_{j + 1}}, \ldots ,h_{\lambda_p}).\]
By (2) of \cite[Lemma 4.1.6]{av}, $\{x \in C : \partial(x) \in \m^2 C\} \subset \m C$. Since $\mu(h_\lambda) \in \m C$, by induction on $p$ we conclude that $\mu(h_{\lambda_1}, \ldots ,h_{\lambda_p})  \in \m C$.

By \cite[Lemma 1.2]{lev1}, we have  $\soc(R) K^R_i \subset (0 : \m^2)\Bo_{i}(K^R)$ for $1 \leq i \leq n - 1$ and so we have $sC \subset (0 : \m^2)\Bo(C)$. Let $z \in \Zi(C)$ be such that $(z, 0) = \partial(y_1, \bar{y}_2)$ for $(y_1, \bar{y}_2) \in \cn(H)$. Then $z = \partial(y_1) + s y_2 \in \Bo(C)$. Therefore, the inclusion $C \hookrightarrow \cn(H)$ induces an injective map $\Ho(C) \hookrightarrow \Ho(\cn(H))$. By abuse of notation we write $(c, 0)$ as $c$. It follows that $\mathfrak{b}_C = \{h_{\lambda}\}_{\lambda \in \Lambda}$ is a linearly independent set in $\Ho_{\geq 1}(\cn(H))$ and $ \mu(h_{\lambda_1}, \ldots ,h_{\lambda_p}) \in \m \cn(H), p \geq 1$ satisfy properties above.

We extend $\mathfrak{b}_C$ to a basis $\mathfrak{b}_{\cn(H)} = \{h_{\lambda}\}_{ \lambda \in \Lambda \sqcup \Lambda'}$ of $\Ho_{\geq 1}(\cn(H))$. Let $h_{\lambda'}$, $\lambda' \in \Lambda'$ be the homology class of $(c_{\lambda'}, \bar{d}_{\lambda'}) \in \Zi_{\geq 1}(\cn(H))$.  Now $\partial(c_{\lambda'}, \bar{d}_{\lambda'}) = 0$ implies that $\partial(c_{\lambda'}) + s d_{\lambda'} = 0$. We have $s d_{\lambda'} = \partial(e_{\lambda'})$ for some $e_{\lambda'} \in (0 : \m^2)C$. We write $(c_{\lambda'}, \bar{d}_{\lambda'}) = (c_{\lambda'} + e_{\lambda'}, 0) + ( - e_{\lambda'}, \bar{d}_{\lambda'})$. Note that $c_{\lambda'} + e_{\lambda'} \in \Zi(C)$. Therefore, after subtracting suitable $R$-linear combinations of $h_{\lambda}$, $\lambda \in \Lambda$ from each $h_{\lambda'}$, $\lambda' \in \Lambda'$ if necessary, we may assume that $h_{\lambda'}$ is a homology class of some cycle in $\displaystyle (0 : \m^2)C \oplus \frac{C}{\m C}$. 

We define $\mu(h_{\lambda'})$, $\lambda' \in \Lambda'$ to be an element in $\displaystyle (0 : \m^2)C \oplus \frac{C}{\m C}$ whose homology class is $h_{\lambda'}$. We extend $\mu$ from $\mathfrak{b}_C^i$ to a Massey operation on $\mathfrak{b}_{\cn(H)}^i$, $i > 1$ such that $\mu : \mathfrak{b}_{\cn(H)}^i \setminus \mathfrak{b}_C^i \rightarrow (0 : \m^2)C$, $i > 1$ by induction on $i$.
Note that $\overline{\mu(h_{\lambda})} \mu(h_{\lambda'}) \in sC \in (0 : \m^2) \Bo(C)$ for $(\lambda, \lambda') \not \in \Lambda \times \Lambda$. We choose $\mu(h_{\lambda}, h_{\lambda'}) \in (0 : \m^2) C$ such that $\partial(\mu(h_{\lambda}, h_{\lambda'})) = \overline{\mu(h_{\lambda})} \mu(h_{\lambda'})$. 
%We define $\mu(h_{\lambda'}, h_{\lambda}) \in (0 : \m^2)C$ for $\lambda \in \Lambda$, $\lambda' \in \Lambda'$ similarly. 
Now assume that $\mu(h_{\delta_1}, \ldots, h_{\delta_i}) \in (0 : \m^2)C$, $(h_{\delta_1}, \ldots, h_{\delta_i}) \in \mathfrak{b}_{\cn(H)}^i \setminus \mathfrak{b}_C^i$ satisfying desired relations are constructed for all $i \leq p$. We choose $(h_{\delta_1}, \ldots , h_ {\delta_{p+1}}) \in \mathfrak{b}_{\cn(H)}^{p+1} \setminus \mathfrak{b}_C^{p+1}$ and observe that $\sum_{j = 1}^{p}\overline{\mu(h_{\delta_1}, \ldots ,h_{\delta_j})}\mu(h_{\delta_{j + 1}}, \ldots ,h_{\delta_{p+1}})$ is an element in $sC$. Therefore, we can choose $\mu(h_{\delta_1}, \ldots ,h_{\delta_{p+1}}) \in (0 : \m^2)C$ such that $\partial(\mu(h_{\delta_1}, \ldots ,h_{\delta_{p+1}})) = \sum_{j = 1}^{p}\overline{\mu(h_{\delta_1}, \ldots ,h_{\delta_j})}\mu(h_{\delta_{j + 1}}, \ldots ,h_{\delta_{p+1}})$. Thus by induction $\mu$ extends to a trivial Massey operation on $\mathfrak{b}_{\cn(H)}$. Therefore, $\cn(H)$ is a Golod algebra and the result follows.
\end{proof}

%\begin{corollary}
%Let $(R, \m)$ be a Gorenstein local ring and $K^R$ denote its Koszul complex. Let $\edim(R) - \dim(R) = n \geq 2$. Assume the complex $C$ defined by $C_i = K^R_i$ for $0 \leq i \leq n-2$, $C_{n-1} = \frac{K^R_{n-1}}{\Bo_{n-1}(K^R)}$ and $C_i = 0$ for $i \geq n$. Then $R$ has the Backelin-Roos property.
%\end{corollary}
%

%Poincar\'e series of Modules over compressed Gorensteing local rings were studied by Rossi and \c Sega in \cite{rossi}. Under the hypothesis of Theorem \ref{intro1}, they proved that if $\eta : Q \twoheadrightarrow R$ is a minimal Cohen presentation, $I = \ker \eta$, $f \in I \setminus \n^{t+1}$, $t = \max\{ i : I \subset \n^i\}$, then $Q/ (f) \twoheadrightarrow R$ is a Golod homomorphism \cite[Theorem 5.1]{rossi}. They also showed that $R/ \soc(R)$ is a Golod ring \cite[Proposition 6.3]{rossi}. We use the ideas developed by them and our Theorem \ref{sbr4.5} to produce a simple and short proof of their main result \cite[Theorem 5.1]{rossi}. Note that our result is stronger.  Our theorem below shows that any choice of $f \in I \setminus \n I$ is enough.
%

We recall the definition of compressed Gorenstein local rings.

\begin{definition}\label{comp}
Let $(R,\m,k)$ be an Artinian Gorenstein local ring of Loewy length $s$ and embedding dimension $n \geq 2$. Set 
$\varepsilon_i = \min \Bigg \{ \dbinom{n-1+s-i}{n-1}, \dbinom{n-1+i}{n-1} \Bigg \}$ for all $i$ with $0 \leq i \leq s$. 
Then it is shown in \cite[Proposition 4.2]{rossi}  that  $l(R) \leq \sum_{i = 0}^n \varepsilon_i$. The ring $R$ is called compressed Gorenstein ring if equality holds.
\end{definition}

We provide a different proof of the result of Rossi and \c Sega \cite[Theorem 5.1]{rossi} below.

\begin{theorem}\label{sbr4.7}
Let $R$ be a compressed Gorenstein local ring such that $\edim(R)  = n \geq 2$ and $\lo(R) = s$, $s \geq 2, s \neq 3$. Then $R/ \soc(R)$ is a Golod ring. Consequently,  $R$ satisfies assertions (1), (2) of Theorem \ref{intro2}.
\end{theorem}

\begin{proof}
We follow notations as set in the proof of Theorem \ref{intro2} and Lemma \ref{sbr4.5}. Let $t = \max\{ i : I \subset \n^i\}$. 
It is proved in \cite[Proposition 4.2]{rossi} that $\displaystyle t = \ceil{\frac{s+1}{2}}$; the least integer not less than $\displaystyle \frac{s+1}{2}$.
By \cite[Lemma 1.4]{rossi}, the map $\Ho_{\geq 1}(R/ \m^t \otimes_Q K^Q) \rightarrow \Ho_{\geq 1}(R/ \m^{t-1} \otimes_Q K^Q)$ induced by the surjection $R/ \m^t \twoheadrightarrow R/ \m^{t-1}$ is a zero map. This implies that
$\Zi_{\geq 1}(K^R) \subset  \Bo_{\geq 1}(K^R) + \m^{t-1}K_{\geq 1}^R$ and therefore $\Zi_{\geq 1}(C) \subset  \Bo_{\geq 1}(C) + \m^{t-1}C$. Thus we find a basis $\mathfrak{b} = \{h_{\lambda}\}_{\lambda \in \Lambda}$ of $\Ho_{\geq 1}(C)$ represented by cycles  in $\m^{t-1}C$.

It is proved in \cite[Lemma 4.4]{rossi} that the map $\psi: \Ho_{<n}(\m^{r+1}K^R) \rightarrow \Ho_{<n}(\m^{r}K^R)$ induced by the inclusion $\m^{r+1} \hookrightarrow \m^r$ is zero for $r = s + 1 - t$. Since, $s \geq 2, s \neq 3$, we have $t - 1 \leq r \leq r + 1 \leq 2(t-1)$. This implies that the map  $\Ho_{<n}(\m^{2(t-1)}K^R) \rightarrow \Ho_{<n}(\m^{t-1}K^R)$ is also zero since it factors through $\psi$. Therefore, we have $\Zi_{< n}(\m^{2t - 2}K^R) \subset \Bo(\m^{t-1}K^R)$ which implies $\Zi(\m^{2t - 2}C) \subset \Bo(\m^{t-1}C)$. 
It is worth pointing out that both the lemmas used here are independent of all other results in \cite{rossi}. 

We construct inductively a trivial Massey operation $\mu : \sqcup_{i = 1}^{\infty} \mathfrak{b}^i \rightarrow  \m^{t-1}C$. Define $\mu(h_{\lambda})$ to be a cycle in $\m^{t-1}C$ such that the homology class of $\mu(h_{\lambda})$ is $h_{\lambda}$. Now $\overline{\mu(h_{\lambda})}\mu(h_{\lambda'}) \in \Zi(\m^{2t - 2}C) \subset \Bo(\m^{t-1}C)$, and so we choose $\mu(h_{\lambda}, h_{\lambda'}) \in \m^{t-1}C$ such that $\partial(\mu(h_{\lambda}, h_{\lambda'})) = \overline{\mu(h_{\lambda})}\mu(h_{\lambda'})$.
The method carries over to the next steps of construction. Therefore, $C$ is a Golod DG algebra and the result follows from Lemma \ref{sbr4.5}.
\end{proof}

\begin{lemma}\label{ged3}
Let $(R, \m ,k)$ be an Artinian Gorenstein local ring but not a complete intersection.  Let $\eta : Q \twoheadrightarrow R$ be a minimal Cohen presentation of $R$ and $I = \ker(\eta)$. Assume that $\mu(I) = r$ and $\edim(R)  = n \leq 3$. Then $R/ \soc(R)$ is a Golod ring, so  $R$ satisfies both assertions (1), (2) of Theorem \ref{intro2}.
If $d_R(t) = 1 - rt^2 - rt^3 + t^5$, then for any finitely generated $R$-module $M$, we have $d_R(t) P^R_M(t) \in Z[t]$. 
The Poincar\'e series of $k$ is given by 
\[P^R_k(t) = \frac{(1 + t)^n}{1 - rt^2 - rt^3 + t^5}.\]
\end{lemma}
\begin{proof}
As before, we follow notations as set in the proof of Theorem \ref{intro2} and Lemma \ref{sbr4.5}.
By \cite[Satz 7]{wiebe} we have $\Ho_1(K^R)^2 = 0$ giving $\Ho_1(C)^2 = 0$.  For a proof written in English, we refer to \cite[Corollary 3.4.8]{bruns}.
Now   $C$ is a DG algebra of length $2$. 
Therefore, any basis of $\Ho_{\geq1}(C)$ admits a trivial Massey operation and so $C$ is a Golod algebra. The first part of the result follows from Lemma \ref{sbr4.5}. 

We compute the denominator. The Koszul complex of $R$ is of length three. We observe that $\dim_k \Ho_0(K^R) = 1$; $\dim_k \Ho_1(K^R) = \mu(I) = r$; $\dim_k \Ho_3(K^R) = \dim_k \soc(R) = 1$ and so $\dim_k \Ho_2(K^R) =  r$ since $\sum_{i \geq 0} (-1)^i \dim_k \Ho_i(K^R)$ must be zero. With the notations used in Theorem \ref{intro2}, we have $P^Q_R(t) = \sum_{i = 0}^3 \dim_k \Ho_i(K^R) = 1 + rt + rt^2 + t^3$. Therefore, we have 
\[d_R(t) = 1 - t(P^Q_R(t) - 1) + t^{n+1}(1+t)
= 1 - t^2(r + rt + t^2) + t^4(1 + t)
= 1 - rt^2 - rt^3 + t^5.
\]
The formula for $P^R_k(t)$ follows from a result of Rossi and \c Sega in \cite[Proposition 6.2]{rossi}.
\end{proof}

If $(R, \m ,k)$ is a Gorenstein local ring satisfying the hypothesis of the theorem below, then the rational expression of $P^R_k(t) $ was computed by Wiebe, cf. \cite[Satz 9]{wiebe}. Later Avramov et al. proved that if $(R, \m)$ is any Artinian ring such that $\edim(R) - \depth(R) \leq 3$, then all $R$-modules have rational Poincar\'e series, cf. \cite[Theorem 6.4]{av4}. We prove a weaker version.

%The following is originally due to Wiebe \cite[Satz 9]{wiebe} and Avramov et al. \cite[Theorem 6.4]{av4}.
\begin{theorem}\label{ged4}
Let $(R, \m ,k)$ be a Gorenstein local ring but not a complete intersection such that $\edim(R) - \depth(R) = n \leq 3$. Let $\eta : Q \twoheadrightarrow \hat{R}$ be a minimal Cohen presentation of $R$, $\ker \eta = I$ and $\mu(I) = r$. Then for any $f \in I \setminus \n I$, the induced map $Q/ (f) \twoheadrightarrow \hat{R}$ is a Golod homomorphism. 

Let $d_R(t) = 1 - rt^2 - rt^3 + t^5$. Then $\displaystyle P^R_k(t) = \frac{(1 + t)^{\edim(R)}}{d_R(t)}$ and for any $R$-module $M$ we have $d_R(t) P^R_M(t) \in \Z[t]$.
\end{theorem}

\begin{proof}

Let $\dim(R) = \depth(R) = d$ and maximal ideal $\m$ be minimally generated by $x_1, \ldots, x_e$ such that $x_1, \ldots, x_d$ form an $R$-sequence. Then $S = R/ (x_1, \ldots, x_d)$ is an Artinian Gorenstein local ring and $\edim(S) = e - d = n \leq 3$. Let $K^R = R \langle X_i : \partial(X_i) = x_i, 1 \leq i \leq n \rangle$ and $K^S = S \langle X_i : \partial(X_i) = x_i, d + 1 \leq i \leq n \rangle$ denote the Koszul complexes. 
The quotient map $q : K^R \twoheadrightarrow K^S$ 
is  a quasi-isomorphism, cf. \cite[Lemma 4.1.6]{av}. 

To prove that the map $Q/ (f) \twoheadrightarrow \hat{R}$ is a Golod homomorphism for any $f \in I \setminus \n I$,  is equivalent to show that for any cycle $z \in \Zi_1(K^R) \setminus \Bo_1(K^R)$, the semi-free extension $K^R\langle T \ | \ \partial(T) = z \rangle$ is a Golod algebra. Now one observes that $q$ extends to a surjective quasi-isomorphism 
\[\tilde{q} : K^R\langle T | \partial(T) = z \rangle \twoheadrightarrow K^S\langle T | \partial(T) = q(z) \rangle,\]
see for instance \cite[Proposition 1.3.5]{gul}. Lemma \ref{ged3} applies to $S$. We conclude that the image of $\tilde{q}$ is a Golod algebra. Therefore, $K^R\langle T \ | \ \partial(T) = z \rangle$ is a Golod algebra.

The statement about Poincare series is an immediate consequence of Lemmas \ref{ds4.5}, \ref{ged3}. 
%It remains only  to show the existence of Golod homomorphism.
\end{proof}

\begin{remark}\label{rem}
In both Theorems \ref{sbr4.7}, \ref{ged4}, we constructed a Golod homomorphism from a hypersurface ring which is a quotient of  an arbitrary generator belonging to a minimal generating set of the defining ideal. Thus both theorems are slightly stronger than their earlier versions. 
\end{remark}

%%%%%%%%%%%%%%%%%%%%%%%%%%%%%%%%%%%%%%%%%%%%%%%%%%%%%
\begin{acknowledgement}
Work on this article started as a result of discussions with Liana M. \c Sega during the conference in honour of Craig Huenke held in Ann Arbor in July 2016. I would like to thank her for clarifications on her results about Auslander-Reiten conjecture. 
I sincerely thank the anonymous referee for  suggesting the proof of Lemma \ref{acg3}. 
%A detailed list of suggestions, corrections advised by the referee improved the exposition in a significant way.
I owe many thanks to Matteo Varbaro for his inspiration and encouragement throughout.
This research is supported by the INdAM-COFUND-2012 Fellowship cofounded by Marie Curie actions, Italy.
 \end{acknowledgement}

\end{document}